\documentclass[margin=1.5in, 11pt]{amsart}

\usepackage[utf8]{inputenc}
\usepackage[margin=1in]{geometry}
\usepackage{amsmath, amssymb, amsthm, verbatim, hyperref, xcolor,multirow,fancyvrb, tikz-cd,ytableau}
\usetikzlibrary{arrows}

\usepackage{thmtools}
\usepackage{thm-restate}

\usepackage{cellspace} %
\declaretheorem[name=Theorem,numberwithin=section]{thm}
\newtheorem{lemma}[thm]{Lemma}
\newtheorem{prop}[thm]{Proposition}
\newtheorem{cor}[thm]{Corollary}
\newtheorem{conj}[thm]{Conjecture}

\theoremstyle{definition}

\newtheorem{construction}[thm]{Construction}
\newtheorem{defn}[thm]{Definition}
\newtheorem{example}[thm]{Example}
\newtheorem{notation}[thm]{Notation}

\theoremstyle{remark}
\newtheorem{remark}[thm]{Remark}

\newcommand{\kk}{\mathbf{k}}

\newcommand{\ZZ}{\mathbb{Z}}
\newcommand{\mm}{\mathfrak{m}}
\newcommand{\cK}{\mathcal{K}}
\newcommand{\cF}{\mathcal{F}}
\newcommand{\cC}{\mathcal{C}}
\newcommand{\cG}{\mathcal{G}}
\newcommand{\nn}{\mathfrak{n}}
\newcommand{\cL}{\mathcal{L}}
\newcommand{\xx}{\mathbf{x}}
\renewcommand{\SS}{\mathbb{S}}
\newcommand{\dd}{\underline{d}}
\newcommand{\FF}{\mathbb{F}}

\newcommand{\In}{\mathrm{in}}
\newcommand{\GL}{\mathrm{GL}}

\newcommand{\into}{\hookrightarrow}
\newcommand{\Tot}{\mathrm{Tot}}
\newcommand{\Hilb}{\mathrm{Hilb}}
\renewcommand{\Im}{\mathrm{Im}}
\newcommand{\Tr}{\mathrm{Tr}}

\title{Syzygies of the transfer ideal of the symmetric group}

\author{Harm Derksen}
\address{Northeastern University, Boston, USA}
\email{ha.derksen@northeastern.edu}

\author{Alexandra Pevzner}
\address{Northeastern University, Boston, USA}
\email{a.pevzner@northeastern.edu}

\date{\today}


\begin{document}

\setcounter{MaxMatrixCols}{18}

\begin{abstract}
    We consider the modular action of the symmetric group $S_n$ on $R = \kk[x_1,\ldots,x_n]$ when $\mathrm{char}(k) = p \leq n$. We show that the image of the transfer map $R\to R^{S_n}$ is an elimination ideal $J\cap R^{S_n}$, where $J\subset R^{S_n}[t]$ is generated by $p$ polynomials with generic coefficients. The structure of this elimination ideal depends only on the quotient $q$ when writing $n = qp + r$ with unique remainder $0 \leq r < p$, implying that the image of the transfer also enjoys this stability. We conjecture a determinantal presentation of the elimination ideal and prove it in the case that $q = 2$. Furthermore, we exhibit a  $\GL$-equivariant, linear minimal free resolution of a certain initial ideal, allowing us to extract the graded Betti numbers of the elimination ideal.
\end{abstract}

\maketitle

\section{Introduction}
    Fix a field $\kk$ and a finite group $G$. Let $G$ act on the polynomial ring $R = \kk[x_1,\ldots,x_n]$ by graded $\kk$-algebra automorphisms. When the action of $G$ is modular, meaning that $|G|$ is not invertible in $\kk$, we cannot average elements of $R$ over the group.
     Instead, the \textit{transfer map}
    \begin{align*}
        \mathrm{Tr}^G: R &\longrightarrow R^G \\
        r &\mapsto \sum_{g\in G} g\cdot r
    \end{align*}
    is a map of $R^G$-modules which can be used to construct invariant polynomials. In the modular case, its image $\Im(\Tr^G)$ is a proper, nonzero ideal of $R^G$ (see \cite[Thm 2.2]{Shank--Wehlau}).

    Taking $G$ to be the symmetric group acting on $R$ by variable permutation, the invariant ring $R^G$ is well known to be the polynomial algebra $\kk[e_1,\ldots,e_n]$ generated by the elementary symmetric polynomials $e_1,\ldots,e_n$. In previous work, the second author conjectured
    that when $\kk = \FF_p$, the minimal free resolution of $\Im(\Tr^{S_n})$ over $R^{S_n}$ depends only on the quotient $q\in\ZZ$ when writing $n$ uniquely as $n = qp + r$ with $0\leq r < p$; see \cite[Conj 6.1]{transfer-paper}. In this paper, we prove the conjecture by realizing the image of the transfer as a module over a polynomial subalgebra of $R^{S_n}$ which, for fixed $q$, is the same for all $0 \leq r \leq p-1$.

\subsection{Outline of paper} In Section \ref{sec: feshbach thm}, we apply a theorem of Feshbach on the radical of $\Im(\Tr^G)$ to express $\Im(\Tr^{S_n})$ as an elimination ideal $\langle f_0, f_1,\ldots, f_{p-1}\rangle \cap R^{S_n}$, where the $f_i$ are certain polynomials in $R^{S_n}[t]$ with generic coefficients. We then note that the structure of the ideal $\langle f_0, f_1,\ldots, f_{p-1}\rangle$ only differs by a change of variables as $n$ varies within $\{qp, qp+1, \ldots, qp + (p-1)\}$. Theorem \ref{thm: ideal stability} makes this idea precise and proves Conjecture 6.1 of \cite{transfer-paper}.

We then consider the elimination ideal $I = \langle f_0, f_1, \ldots, f_{p-1}\rangle \cap R^{S_n}$ in its own right. Conjecture \ref{conj: determinantal ideal is elimination ideal} posits that $I$ is generated by the maximal minors of a matrix $A$, which is built out of concatenations of Sylvester matrix blocks. We show in Proposition \ref{prop: set theoretic description} that $I$ is generated up to radical by sums of maximal minors of $A$. In Section \ref{sec: initial ideal}, we conjecture that the determinantal ideal has a certain squarefree initial ideal with a combinatorial characterization. In Section \ref{sec: d = 2 proof}, we use the initial ideal to prove Conjecture \ref{conj: determinantal ideal is elimination ideal} when $q = 2$. Finally, in Section \ref{sec: d = 2 resolution}, we build explicit, $\GL_{p-1}$-equivariant minimal free resolutions of the associated graded ideal of $I$ when $q = 2$. These resolutions are linear, allowing us to deduce the projective dimension and graded Betti numbers of $I$.

\section{Feshbach's theorem and the transfer ideal as an elimination ideal}\label{sec: feshbach thm}

Let $G$ be a finite group acting on the polynomial ring $R = \kk[x_1,\ldots,x_n]$. Suppose that $\kk = \FF_p$, where the prime $p$ divides $|G|$. We state below a theorem of Feshbach \cite[Thm 2.4]{Feshbach} which describes the radical of $\Im(\Tr^G)$ in terms of group elements of order $p$.

\begin{thm}\cite{Feshbach}\label{thm: feschbach}
    In the setting above,
    $$\sqrt{\Im(\Tr^G)} = \left(\bigcap_{\substack{g \in G: \\ |g| = p}} \langle x_i - g(x_i) : i = 1,\ldots,n \rangle\right) \cap R^G.$$
\end{thm}

If we make the further assumption that $G$ acts on $R$ by variable permutation, then it was shown by Shank--Wehlau \cite[Thm 6.1]{Shank--Wehlau} that $\Im(\Tr^G)$ is a radical ideal, so that Theorem \ref{thm: feschbach} exactly describes $\Im(\Tr^G)$. 

We now specialize to the case that $G$ is the symmetric group $S_n$ with its usual permutation action on $R = \kk[x_1,\ldots,x_n]$. Then $R^{S_n}$ is equal to the polynomial algebra $\kk[e_1,\ldots,e_n]$ generated by the elementary symmetric polynomials, where $\deg(e_i) = i$. In this case, the image of the transfer is an elimination ideal arising from a collection of univariate polynomials, where we now treat the algebraically independent $e_1,\ldots,e_n$ as generic coefficient variables.

\begin{thm}\label{thm: transfer ideal is elimination ideal}
    Fix a positive integer $n$ and let $\kk = \FF_p$, where $p$ is a prime such that $p\leq n$. Consider the standard action of the symmetric group $S_n$ on the polynomial ring $R = \kk[x_1,\ldots,x_n]$. Write $n = qp + r$, where $0 \leq r < p$ and $q\geq0$. Then $I_n:= \Im(\Tr^{S_n})$ is an elimination ideal $\langle f^{(r)}_0, f^{(r)}_1,\ldots, f^{(r)}_{p-1}\rangle \cap R^{S_n}$, where the $f^{(r)}_k$ are polynomials in $R^{S_n}[t]$ defined by
    \begin{equation}\label{eq: f_1..f_k}
    f^{(r)}_k =
    \begin{cases}
        t^q + e_pt^{q-1} + e_{2p}t^{q-2} + \cdots + e_{(q-1)p}t + e_{qp} & \text{if }i=0\\
        e_kt^q + e_{p+k}t^{q-1}+\cdots + e_{qp+k} & \text{if }r\geq1\text{ and }1\leq k \leq r \\
        e_kt^{q-1} + e_{p+k}t^{q-2} + \cdots + e_{(q-1)p + k} & \text{if }r < k \leq p-1.
    \end{cases}
    \end{equation}
    We often write $f_k$ for $f_k^{(r)}$ when $r$ is clear from context.
\end{thm}

\begin{proof}
    The elements in the symmetric group which have order $p$ are exactly the $p$-cycles and the products of disjoint $p$-cycles. Fix a $p$-cycle $g = (a_1,\ldots,a_p)$, and consider the prime ideal $$K_g = \langle x_i - g(x_i) : i = 1,\ldots,n\rangle\subset \kk[x_1,\ldots,x_n].$$
    Then $K_g$ vanishes on a point $(x_1,\ldots,x_n)$ in the affine space $\mathbb{A}^n_{\kk}$ if and only if $x_{a_1} = x_{a_2} = \cdots = x_{a_p}$. From this it follows that the vanishing locus of the intersection of all such $K_g$ is given by
    $$V\bigg(\bigcap_{\substack{g\in S_n:\\ |g| = p}} K_g\bigg) = \{(x_1,\ldots,x_n)\in\mathbb{A}^n_\kk : (x_1,\ldots,x_n)\text{ has at least }p\text{ repeated coordinates}\}.$$
    Thus, $(x_1,\ldots,x_n)\in V(\bigcap_gK_g)$ if and only if the polynomial 
    $$F(t) = (t + x_1)(t+x_2)\cdots(t+x_n) \in \kk[t]$$
    has a root of multiplicity at least $p$. If $\alpha\in \kk$ is such a root, then we can write $$F(t) = (t-\alpha)^p\cdot G(t) = (t^p - \alpha^p)\cdot G(t)$$ for some $G(t)\in \kk[t]$. Collecting coefficients on powers of $t$ based on their residue mod $p$, one can uniquely express $F(t)$ as
    \begin{equation}\label{eq: t^p}
        F(t) = \sum_{k=0}^{p-1}t^k \tilde{f}_k(t^p)
    \end{equation} for some polynomials $\tilde{f}_0, \tilde{f}_1,\ldots,\tilde{f}_{p-1}$. Equation (\ref{eq: t^p}) must agree with the expression for $F(t)$ in terms of the elementary symmetric polynomials in $x_1,\ldots,x_n$, i.e. 
    $$F(t) = t^n + e_1t^{n-1} + e_2t^{n-2} + \cdots + e_n.$$
    Thus, the $\tilde{f}_k$ are exactly the polynomials $f_k$ given in the statement of the theorem. Moreover, the $f_k$ must have a common root (namely $\alpha^p$). It then follows from Theorem \ref{thm: feschbach} that $I_n$ is exactly the elimination ideal $\langle f_0, f_1, \ldots, f_{p-1}\rangle \cap R^{S_n}$, as desired. 
\end{proof}

\section{Stability in the syzygies of the transfer ideal}\label{sec: syz stability}

It was conjectured in \cite[Conj 6.1]{transfer-paper} that the minimal free resolution of $I_n:= \Im(\Tr^{S_n})$ as an $\FF_p[e_1,\ldots,e_n]$-module depends only on $q$ (up to predictable degree shifts) when writing $n$ uniquely as $n = qp + r$, with $q,r\in\ZZ_{\geq0}$ and $r \leq p-1$. This can be proven using the realization of $I_n$ as the elimination ideal of Theorem \ref{thm: transfer ideal is elimination ideal}. In fact, fixing $q$, the ideals $I_{qp+r}$ have the same structure for all $0 \leq r \leq p-1$, which we make precise in the next theorem.

\begin{thm}\label{thm: ideal stability}
    Fix $q\geq 0$ and let $A = \FF_p[e'_1, \ldots, e'_{qp}]$ be the $\ZZ/p\ZZ$-graded polynomial ring with $\deg(e'_i) = i\bmod p$. Let the ring of symmetric polynomials $R^{S_{qp+r}} = \FF_p[e_1,\ldots,e_{qp+r}]$ also have the $\ZZ/p\ZZ$-grading $\deg(e_i) = i\bmod p$. Then, there exist homogeneous ring inclusions $\iota_r:A \into R^{S_{qp+r}}$ for each $0\leq r \leq p-1$ such that
    \begin{enumerate}
        \item[(i)] each $\iota_r$ is a flat ring extension, and
        \item[(ii)] via the map $\iota_r$, there is an isomorphism $R^{S_{qp+r}}\otimes_A I_{qp}\cong I_{qp+r}$ of $A$-modules.
    \end{enumerate}
\end{thm}

\begin{proof}
    Fix $0 \leq r \leq p-1$ and let $n = qp+r$. Let $J$ be the ideal $\langle f_0,f_1, \ldots, f_{p-1}\rangle$ of $R^{S_n}[t]$, where the $f_i$ are as in (\ref{eq: f_1..f_k}).
    For each $1 \leq i \leq r$, we may replace $f_i$ by the $S$-pair $$f_i' := S(f_0, f_i) = f_i - e_if_0$$ of degree $q-1$ so that the ideals $\langle f_0, f_1,\ldots,f_{p-1}\rangle$ and $\langle f_0, f_1', \ldots, f'_r, f_{r+1}, \ldots, f_{p-1}\rangle$ are equal. The coefficient on $t^{q-k}$ in $f_i'$ is $e_{kp+i} - e_{kp}e_i$. 
    
    We define the inclusion $\iota_r:A \to R^{S_n}$ by
    $$
    \iota_r(e_j') = 
    \begin{cases}
        e_j & \text{if }j\equiv i\bmod p, \text{ where }i = 0\text{ or }r < i \leq p-1\\
        e_{(d+1)p + i} - e_{(d+1)p}e_i & \text{if }j = dp+i,\text{ where }1 \leq i \leq r
    \end{cases}.
    $$
    Then, the ideal $\langle f_0, f_1, \ldots, f_{p-1}\rangle\subset R^{S_n}[t]$ is equal to the ideal $\langle g_0, g_1,\ldots,g_{p-1}\rangle$, where $$g_i = f^{(0)}(\iota_r(e_1'), \ldots, \iota_r(e_{qp}'))$$
    as in (\ref{eq: f_1..f_k}). In particular, the elimination ideals $\langle f_0, f_1,\ldots, f_{p-1}\rangle\cap R^{S_n}$ and $\langle g_0, g_1,\ldots,g_{p-1}\rangle\cap R^{S_n}$ are also equal. Moreover, $I_{qp+r}$ is the extension of the ideal $$I_{qp}=\langle f_0^{(0)}(e_1', \ldots, e_{qp}'),\ldots,f^{(0)}_{p-1}(e_1', \ldots, e_{qp}')\rangle \cap A$$ from $A$ to $S^{qp+r}$, via the map $\iota_r$.
    
    We now show that $R^{S_n}$ is a flat $A$-algebra. Consider the $\mathbb{N}$-grading on $R^{S_{n}}$ where $\deg(e_i) = i$. We may put an $\mathbb{N}$-grading on $A$ with $\deg(e_i') := \deg(\iota_r(e_i'))$, so that $\iota_r$ is homogeneous. Note that this grading on $A$ depends on $r$, but after taking degrees mod $p$ there is no dependence on $r$. Now since $A$ is $\mathbb{N}$-graded with $A_0 = \kk$ and $R^{S_n}$ is an $\mathbb{N}$-graded $A$-module, by \cite[Exercise 6.10]{EisenbudCA}, $R^{S_n}$ is a flat $A$-module if and only if the localization $R^{S_n}_P$ is flat, where $P = \langle e_1',\ldots,e_{qp}'\rangle$ is the homogeneous maximal ideal of $A$. By \cite[Exercise 18.18]{EisenbudCA}, $R^{S_n}_P$ is flat since the images of $e_1',\ldots,e_{qp}'$ under $\iota_r$ form a regular sequence. By flatness, the map $R^{S_n}\otimes_A I_{qp} \to I_{n}$ is an isomorphism.  
\end{proof}

\begin{cor}\label{cor: syzygy stability}
    Fixing $q\geq0$, let $\cF$ be a graded minimal free resolution of $I_{qp}$ over $A$. Then for each $0\leq r \leq p-1$, the complex $R^{S_{qp+r}}\otimes_{A}\cF$ gives a graded minimal free resolution of $I_{qp+r}$ over $R^{S_{qp+r}}$. In particular, the $\ZZ/p\ZZ$-graded Betti numbers of $I_{qp+r}$ over $R^{S_{qp+r}}$ are the same for all $r$.
\end{cor}

\begin{example}
    Consider $p = 3$ and $q = 2$. For each $r = 0,1,2$, let $f_0^{(r)}, f_1^{(r)}, f_2^{(r)}\in R^{S_{qp+r}}[t]$ be as in (\ref{eq: f_1..f_k}), so that $$I_{qp+r} = I_{6+r} = \langle f_0^{(r)},f_1^{(r)}, f_2^{(r)}\rangle \cap R^{S_{6+r}}.$$ These polynomials are listed below.
    
    \begin{minipage}{0.3\textwidth}
    \begin{align*}
        f_0^{(0)} &= t^2 + e_3t + e_6\\
        f_1^{(0)} &= e_1t + e_4 \\
        f_2^{(0)} &= e_2t + e_5
    \end{align*}
    \end{minipage}
    \begin{minipage}{0.35\textwidth}
       \begin{align*}
        f_0^{(1)} &= t^2 + e_3t + e_6\\
        f_1^{(1)} &= e_1t^2 + e_4t + e_7 \\
        f_2^{(1)} &= e_2t + e_5
    \end{align*} 
    \end{minipage}
    \begin{minipage}{0.35\textwidth}
        \begin{align*}
        f_0^{(2)} &= t^2 + e_3t + e_6\\
        f_1^{(2)} &= e_1t^2 + e_4t + e_7 \\
        f_2^{(2)} &= e_2t^2 + e_5t + e_8
    \end{align*}
    \end{minipage}
We now apply the replacements $f_k^{(r)}\mapsto f_k'^{(r)}$ as in the proof of Theorem \ref{thm: ideal stability}. The polynomials which have been replaced are indicated in bold.

\begin{minipage}{0.3\textwidth}
    \begin{align*}
        f_0^{(0)} &= t^2 + e_3t + e_6\\
        f_1^{(0)} &= e_1t + e_4 \\
        f_2^{(0)} &= e_2t + e_5
    \end{align*}
\end{minipage}
\begin{minipage}{0.35\textwidth}
    \begin{align*}
        f_0^{(1)} &= t^2 + e_3t + e_6\\
        \boldsymbol{f_1'^{(1)}} &= \boldsymbol{(e_4-e_1e_3)t + (e_7-e_1e_6)} \\
        f_2^{(1)} &= e_2t + e_5
    \end{align*}
\end{minipage}
\begin{minipage}{0.35\textwidth}
    \begin{align*}
        f_0^{(2)} &= t^2 + e_3t + e_6\\
        \boldsymbol{f_1'^{(2)}} &= \boldsymbol{(e_4-e_1e_3)t + (e_7-e_1e_6)} \\
        \boldsymbol{f_2'^{(2)}} &= \boldsymbol{(e_5 - e_2e_3)t + (e_8 - e_2e_6)}
    \end{align*}
\end{minipage}

\vspace{5pt}

The maps $\iota_r:\kk[e_1', \ldots, e_6'] \to \kk[e_1,\ldots,e_{6 + r}]$ are then given by the following matrices.

\vspace{5pt}

\begin{minipage}{0.3\textwidth}
    $$\begin{pmatrix}
        e'_1 \\ e'_2 \\ e'_3 \\ e'_4 \\ e'_5 \\ e'_6
    \end{pmatrix}    \overset{\iota_0}{\longmapsto}
    \begin{pmatrix}
        e_1 \\ e_2 \\ e_3 \\ e_4 \\ e_5 \\ e_6
    \end{pmatrix}$$
\end{minipage}
\begin{minipage}{0.35\textwidth}
    $$\begin{pmatrix}
        e'_1 \\ e'_2 \\ e'_3 \\ e'_4 \\ e'_5 \\ e'_6
    \end{pmatrix}    \overset{\iota_1}{\longmapsto}
    \begin{pmatrix}
        e_4 - e_1e_3 \\ e_2 \\ e_3 \\ e_7 - e_1e_6 \\ e_5 \\ e_6
    \end{pmatrix}$$
\end{minipage}
\begin{minipage}{0.35\textwidth}
    $$\begin{pmatrix}
        e'_1 \\ e'_2 \\ e'_3 \\ e'_4 \\ e'_5 \\ e'_6
    \end{pmatrix}    \overset{\iota_2}{\longmapsto}
    \begin{pmatrix}
        e_4 - e_1e_3 \\ e_5 - e_2e_3 \\ e_3 \\ e_7 - e_1e_6 \\ e_8 - e_2e_6 \\ e_6
    \end{pmatrix}$$
\end{minipage}
\end{example}

\section{Generators as minors of a matrix}\label{sec: determinantal conjecture}
We now consider the elimination ideal $I = \langle f_0,f_1,\ldots,f_{p-1}\rangle \cap \kk[e_1,\ldots,e_n]$ in its own right. In particular, we let $\kk$ be any field and let $p\geq3$ be any integer. Thus, we simply consider the conditions on the coefficients $(e_1,\ldots,e_{n})$ of $f_0,\ldots,f_{p-1}$ which would imply that the $f_i$ share a common root.

The assumptions that $p$ is prime and that $\kk = \FF_p$ were needed to establish that the image of the transfer is equal to the elimination ideal $I$. However, the arguments in the proof of Theorem \ref{thm: ideal stability} do not require these assumptions. Hence, when considering the elimination ideal $I$, it suffices to consider $n = qp$ for $q\geq0$. Thus, we have $f_0$ monic of degree $q$ and $f_1,\ldots,f_{p-1}$ generic of degree $q-1$ as in (\ref{eq: f_1..f_k}). 

Let $A_0, A_1, \ldots, A_{p-1}$ be the matrices
\begin{equation}\label{eq: sylvester matrices}
    A_0 :=
\begin{pmatrix}
    1  & 0 & \cdots & 0\\
    e_{p} & 1 & \cdots & 0\\
    e_{2p}  & e_{p} & \cdots & 0\\
    \vdots & \vdots & \ddots  & \vdots \\
    e_{qp}  & e_{(q-1)p} &  \cdots & 1\\
    0  & e_{e_{qp}}&  \cdots & \vdots\\
    \vdots  & \vdots & \cdots & e_{(q-1)p} \\
    0 & 0 & \cdots & e_{qp}
\end{pmatrix} \qquad 
A_i :=
\begin{pmatrix}
    e_{i}  & 0 & \cdots & 0\\
    e_{p+i} & e_{i} & \cdots & 0\\
    e_{2p+i}  & e_{p+i} & \cdots & 0\\
    \vdots & \vdots & \ddots  & \vdots \\
    e_{(q-1)p + i}  & e_{(q-2)p + i} &  \cdots & e_{i}\\
    0  & e_{(q-1)p+i}&  \cdots & \vdots\\
    \vdots  & \vdots & \cdots & e_{(q-2)p + i} \\
    0 & 0 & \cdots & e_{(q-1)p + i}
\end{pmatrix}
\end{equation}
where $A_0$ has $q-1$ columns and all other $A_i$ have $q$ columns. Then, for each $i\neq0$, the block matrix $[A_0 \,\,\, A_i]$ is the Sylvester matrix whose determinant computes the resultant of $f_0$ and $f_i$. Let $A$ be the $(2q-1)\times ((q-1) + (p-1)q)$ block matrix
\begin{equation}\label{eq: matrix A}
    A:= \big[A_0 \,\,\, A_1 \,\, \cdots \,\, A_{p-1}\big].
\end{equation}
Let $J$ denote the ideal of maximal minors of $A$. We first observe that any maximal minor of $A$ must vanish whenever $f_0, f_1, \ldots, f_{p-1}$ have a common root.

\begin{prop}\label{prop: J contained in I}
    There is a containment $\sqrt{J}\subseteq \sqrt{I}$. When $\kk =\FF_p$ for $p$ prime, then $\sqrt{J}\subseteq I$.
\end{prop}

\begin{proof}
    Suppose $(e_1,\ldots,e_{qp})\in V(I)\subset \kk^{qp}$, so that the polynomials $f_0,\ldots,f_{p-1}$ have some common root $\alpha\in\kk$. Let $v$ be some column of $A$. Then the nonzero entries of $v$ are the coefficients of $f_i$ for some $i$, and the dot product of $v$ with the vector $w:=\left[1, \alpha, \alpha^2,\ldots, \alpha^{2q-2}\right]$ is $\alpha^jf_i(\alpha)$ for some $j$, so it is equal to 0. Hence, given any square submatrix of $A$, the linear combination of its rows with coefficients equal to the entries of $w$ is zero. This submatrix therefore does not have full rank and its determinant is zero. This shows that $V(I)\subseteq V(J)$, i.e. $\sqrt{J}\subseteq \sqrt{I}$. When $p$ is prime and $\kk = \mathbb{F}_p$, the ideal $I$ is equal to the image of the transfer $I_{qp}$, which is radical by \cite[Thm 6.1]{Shank--Wehlau}.
\end{proof}

\begin{conj}\label{conj: determinantal ideal is elimination ideal}
    The elimination ideal $I$ is equal to the determinantal ideal $J = \mathcal{I}_{2q-1}(A)$.
\end{conj}

This conjecture has been verified for $p = 3$ with $q\leq 6$ and for some larger values of $p$ with limited $q$.

\subsection{Set theoretic description}

We show that when $\kk$ is algebraically closed, the elimination ideal $I$ is generated up to radical by certain sums of the minors of $A$.

\begin{lemma}
    The polynomials $f_0, f_1, \ldots, f_{p-1}$ as in (\ref{eq: f_1..f_k}) have a common root in the algebraic closure $\overline{\kk}$ if and only if the polynomials $f_0$ and $h_\mu$ have a common root in $\overline{\kk}$ for all $\mu = (\mu_1,\ldots\mu_{p-1})\in \overline{\kk}^{p-1}$, where 
    $$h_\mu = \sum_{i=1}^{p-1} \mu_if_i.$$
\end{lemma}

\begin{proof}
    The forward direction is clear. We prove the contrapositive of the backward direction. Let 
    $$V = \left\{(\mu_1,\ldots,\mu_{p-1}):h_\mu, f_0 \text{ have a common root }\right\}\subset \overline{\kk}^{p-1}.$$
    Let $\alpha_1, \ldots, \alpha_q\in\overline{\kk}$ be the roots of $f_0$. Then $\mu\in V$ if and only if $h_\mu(\alpha_i) = 0$ for some $\alpha_i$. Thus, 
    $V = \bigcup_{i=1}^q V_i$, where
    $$V_i = \{\mu\in\kk^{p-1} : h_\mu(\alpha_i) = 0\}.$$
    If $f_0, f_1, \ldots, f_{p-1}$ have no common root, then for each $i$, the space $V_i$ is a hyperplane in $\overline{\kk}^{p-1}$ defined by the nonzero linear condition
    $$f_1(\alpha_i)\mu_1 + f_2(\alpha_i)\mu_2 + \cdots + f_n(\alpha_i)\mu_{p-1} = 0.$$
    Since $\overline{\kk}$ is infinite, one can find a point $\mu\notin\overline{\kk}^{p-1}\setminus V$.
\end{proof}

Hence, $f_0,f_1,\ldots,f_{p-1}$ have a common root if and only if the resultant $R(f_0, h_\mu)$ vanishes for all $\mu\in\kk^{p-1}$. This resultant is a homogeneous, degree $q$ polynomial $\sum g_\alpha\mu^\alpha$ in $\mu_1,\ldots,\mu_{p-1}$ which vanishes only if each coefficient $g_\alpha$ vanishes. 

For $1 \leq i \leq {p-1}$ and $1 \leq j \leq q$, consider the column vector 
$$v_{i,j} = (\underbrace{0, \ldots, 0}_{i-1}, e_i, e_{p+i}, \ldots e_{(q-1)p + i}, \underbrace{0, \ldots, 0}_{q-i})^T.$$
In other words, $v_{i,j}$ is the $j^{\text{th}}$ column vector of the matrix $A_i$ as in (\ref{eq: sylvester matrices}). Similarly, for $i = 0$ and $1\leq j \leq q-1$, let $v_{0,j}$ be the column vector
$$v_{0,j} = (\underbrace{0,\ldots,0}_{i-1}, 1, e_p, e_{2p},\ldots, e_{qp}, \underbrace{0, \ldots, 0}_{q-1-i})^T,$$
so that the resultant $R(f_0, f_i)$ (for $1\leq i \leq p-1$) is
$$\det\left(v_{0,1} \, \cdots \, v_{0, q-1} \, v_{i,1}\, \ldots, v_{i,q}\right).$$

\begin{prop}\label{prop: set theoretic description}
    For a monomial $\mu^\alpha$ of degree $q$, the coefficient of $\mu^\alpha$ in $R(f_0, h_\mu)$ is
    \begin{equation}\label{eq: mixed resultants}
    \sum_\beta \det(v_{0,1} \, \cdots\, v_{0,q-1} \, v_{\beta_1,1} \, v_{\beta_2,2}\, \cdots \, v_{\beta_q,q}),
    \end{equation}
    where the sum is taken over all distinct rearrangements $(\beta_1,\ldots,\beta_q)$ of the integer vector$$(\underbrace{1,\ldots,1}_{\alpha_1\text{ times}}, \underbrace{2,\ldots,2}_{\alpha_2\text{ times}}, \ldots, \underbrace{p-1,\ldots,p-1}_{\alpha_{p-1}\text{ times}}).$$
\end{prop}

\begin{proof}
    We identify a determinant as in (\ref{eq: mixed resultants}) with an element of the exterior power $\bigwedge^{2q-1}V$, where $V$ is the $\kk$-vector space with basis $\{v_{i,j}\}$ as described above. Then the resultant $R(f_0, h_\mu)$ is identified with the element
    \begin{align*}
        \left(v_{0,1}\wedge \cdots \wedge v_{0, q-1}\right) \wedge
    \left(\mu_1v_{1,1} + \cdots + \mu_{p-1} v_{p-1,1} \right) \wedge
    \left(\mu_1 v_{1,2} + \cdots + \mu_{p-1}v_{p-1,2}\right) \wedge \\\cdots \wedge
    \left(\mu_1 v_{1,q} + \cdots + \mu_{p-1} v_{p-1,q}\right).
    \end{align*}
    
    Using multilinearity and expanding in terms of the basis of wedges of the $v_{i,j}$ yields the result.
\end{proof}

\begin{cor}
    The elimination ideal $I = \langle f_0, \ldots, f_{p-1}\rangle \cap \kk[e_1,\ldots,e_{qp}]$ agrees up to radical with the ideal generated by the polynomials (\ref{eq: mixed resultants}), as $\alpha$ ranges over all monomials in $\mu_1,\ldots,\mu_{p-1}$ of degree $q$.
\end{cor}

\section{Initial terms of the determinantal ideal}\label{sec: initial ideal}

Let $\kk$ be any field, and let $p\geq3$ be any integer (not necessarily prime). Consider $f_0,\ldots,f_{p-1}$ as in (\ref{eq: f_1..f_k}). In this section, we will show that the ideal
\begin{equation}\label{eq: def of initial ideal}
L := \bigcap_{j=1}^{(q-1)p + 1}
\langle e_j, e_{j+1}, \ldots, e_{j+p-2} \rangle
\end{equation}
is contained in the initial ideal of $I$ under the grevlex order on $S:=\kk[e_1,\ldots,e_{qp}]$. More specifically, we show that any maximal minor of the matrix $A$ from (\ref{eq: matrix A}) lies in $L$. We first note a nice combinatorial characterization of when a monomial of $S$ lies in $L$.

\begin{defn}\label{def: large gap}
    Let $m = e_1^{a_1}\cdots a_{qp}^{a_{qp}}\in S$ be a monomial. Say that $m$ has a \textit{large gap} if its support vector $(a_1,\ldots,a_{qp})$ has at least $p-1$ consecutive $0$ entries.
\end{defn}

\begin{lemma}\label{lem: gap monomials}
    A monomial $m = e_1^{a_1}\cdots e_{qp}^{a_{qp}}$ lies in $L$ if and only if $m$ does not have a large gap.
\end{lemma}

\begin{proof}
    Suppose that $m$ has a large gap. Then there is some $j$ such that $a_k = 0$ for all $j \leq k \leq j +p-2$, meaning that $m\notin \langle e_j, e_{j+1},\ldots, e_{j+p-2}\rangle$. Conversely, suppose that $m$ does not have a large gap. Then given any consecutive portion $(a_j, a_{j+1}, \ldots, a_{j+p-2})$ of its exponent vector, there must be some positive entry $a_k>0$ with $j \leq k \leq j+p-2$. Hence, $e_k$ divides $m$, and $m\in \langle e_j, e_{j+1},\ldots,e_{j+p-2}\rangle$. This is true for all $j = 1,\ldots,(q-1)p+1$, so $m\in L$.
\end{proof}

To conveniently describe initial terms of minors of the matrix $A$, we rearrange the columns of $A$ to a matrix $A'$ so that in each row, the variables $e_i$ appear in increasing order of their indices $i$. This forces the 0 entries to be left-justified on the top of the matrix and right-justified on the bottom. Specifically, we define the matrix $A'$ to be the $(2q-1)\times ((q-1) + (p-1)q)$ matrix with entries
\begin{equation}\label{eq: A'}
    A'_{i,j} = pi + j - qp,
\end{equation}
 where we take $e_0 = 1$ and $e_k=0$ whenever $k < 0$ or $k > qp$. For example, when $p = 4$ and $q = 3$, one has
$$
A' = \begin{pmatrix}
    0 & 0 & 0 & 0 & 0 & 0 & 0 & 1 & e_1 & e_2 & e_3\\
    0 & 0 & 0 & 1 & e_1 & e_2 & e_3 & e_4 & e_5 & e_6 & e_7 \\
    e_1 & e_2 & e_3 & e_4 & e_5 & e_6 & e_7 & e_8 & e_9 & e_{10} & e_{11} \\
    e_5 & e_6 & e_7 & e_8 & e_9 & e_{10} & e_{11} & e_{12} & 0 & 0 & 0 \\
    e_9 & e_{10} & e_{11} & e_{12} & 0 & 0 & 0 & 0 & 0 & 0 & 0
\end{pmatrix}
$$

\begin{prop}\label{prop: antidiagonal term has no gap}
    Let $S = \kk[e_1,\ldots,e_{qp}]$ and consider the matrix $A'$ as above. Fix a $(2q-1)\times(2q-1)$ submatrix $B$ of $A$. Then,
    \begin{enumerate}
        \item[(i)]if the product of the antidiagonal entries of $B$ is 0, then $\det(B) = 0$, and
        \item[(ii)] the product of the antidiagonal entries of $B$ lies in $L$, where $L$ is as defined in (\ref{eq: def of initial ideal}).
    \end{enumerate}
\end{prop}

\begin{proof}
    For (i), suppose that $B$ has a 0 entry in position $(i, 2q-i)$ on its antidiagonal. Then by construction of $A'$, $B_{k,\ell} = 0$ for every $(k,\ell)$ such that $k\leq i$ and $\ell \leq 2q-i$. Hence, $B$ has the form
    $$
    B = \begin{pmatrix}
        \multicolumn{3}{ c }{0} & \Bigg\vert & 
        \multicolumn{3}{c}{C} \\
        \hline
        0 & \cdots & 0 & \big\vert & * & \cdots & * \\ \hline
        \multicolumn{7}{c}{D}
    \end{pmatrix}
    $$
    where $C$ is a square $(i-1)\times(i-1)$ matrix. Then, expanding along the first row, for example, one sees that $\det(B) = \det(C)\cdot 0 = 0$. 
    
    For (ii), suppose that $\det(B)\neq0$. By (i), all of the antidiagonal entries of $B$ are nonzero. By construction of $A'$, if $B_{i,j} = e_k$ and $B_{i+1,j-1} = e_\ell$, then either $\ell \leq k$ or $0 < \ell - k \leq p-1$. In other words, when comparing antidiagonal entries of $B$ in consecutive rows, the index of the entry in the later row can increase by at most $p-1$. Since the antidiagonal entry from the first row must lie in $\{1 = e_0, e_1, \ldots, e_{p-1}\}$, and the antidiagonal entry from the last row must lie in $\{e_{(q-1)p + 1}, e_{(q-1)p + 2}, \ldots, e_{qp}\}$, there must be no gaps of size $>p-1$ between variables $e_1, \ldots, e_{qp-1}$. Note that if one of the antidiagonal entries is $e_{qp}$, then there must be another antidiagonal entry $e_k$, where $(q-1)p+1\leq k \leq (q-1)p + (p-1)$. Similarly if one of the antidiagonal entries is $1 = e_0$, then there must be another entry $e_\ell$ with $1 \leq \ell \leq p-1$. Hence, the product of the antidiagonal entries lies in $L$.
\end{proof}

Given a $(2q-1)\times (2q-1)$ submatrix $B$ of $A'$, the monomials in its determinant correspond to permutations $w = (w_1,\ldots,w_{2q-1})$ via $w \leftrightarrow m_w :=\prod_{i=1}^{2q-1} B_{i, w_i}$. Consider the action of the symmetric group on itself via position swaps. Denote by $t_{i,j}$ the swap of positions $i$ and $j$.

\begin{lemma}\label{lem: smaller inversion implies smaller in grevlex}
    Let $w\in S_{2q-1}$ be a permutation with $m_w\neq0$, and let $i < j$. If $t_{i,j}w$ has a smaller inversion number than $w$, then $m_{t_{i,j}w} \prec m_w$ in the grevlex order on $S = \kk[e_1,\ldots,e_{qp}]$.
\end{lemma}

\begin{proof}
    Choose $i,j$ such that $i< j$ and $w_i > w_j$. Then $t_{i,j}w$ has a smaller inversion number than $w$. If $m_{t_{i,j}w} = 0$ then we are done, so assume this is not the case. If $B_{i, w_i} = e_k$ and $B_{j,w_j} = e_\ell$, then there is some $m$ for which $B_{i, w_j} = e_{\ell - mp}$ and $B_{j, w_i} = e_{k + mp}$. Diagrammatically, $B$ has the following form
    $$
    \begin{pmatrix}
        \, & e_{\ell-mp} & \, & e_k & \,\\
        \, & \, & \, & \, &\, \\
        \, & e_\ell & \, & e_{k+mp}&\,\\
    \end{pmatrix}.
    $$
    Hence, there is a monomial $m'$ such that \begin{align*}
        m_w &= m' \cdot e_ke_\ell\\
        m_{t_{i,j}w} &= m' \cdot e_{\ell - mp}e_{k + mp}.
    \end{align*}
    It suffices to compare $e_ke_\ell$ and $e_{\ell-mp}e_{k+mp}$. But since $e_{k+mp}$ is in the same row and to the right of $e_\ell$ in $M$, we have $k+mp>\ell$, therefore in the grevlex order, $e_{\ell-mp}e_{k+mp} \prec e_ke_\ell$, as desired.
\end{proof}

\begin{prop}\label{prop: lead term is antidiagonal}
    The leading term of any maximal minor of $A'$ under grevlex order is the antidiagonal term.
\end{prop}

\begin{proof}
    Fix a $(2q-1)\times (2q-1)$ submatrix $B$ of $A'$. Then the antidiagonal term of $\det B$ corresponds to the long permutation $w_0 = (2q-1, 2q-2, \ldots, 1)$. We show that given any other nonzero term $m_w$ of $\det B$, we can reach $m_{w_0}$ by applying swaps $t_{i,j}$ such that 
    \begin{enumerate}
        \item[(1)] each $t_{i,j}$ increases the inversion number
        \item[(2)] each $m_{t_{i,j}w}$ is nonzero.
    \end{enumerate}
    It would then follow from Lemma \ref{lem: smaller inversion implies smaller in grevlex} that $m_{w_0}$ is the leading term of $\det{B}$.
    Fix a permutation $w$, and let $a$ be the the smallest index at which $w$ and $w_0$ differ. Precisely,
    $$a = \min\{i : w(i) \neq n-i+1\}.$$ Then, $w(b) = n-a+1$ for some index $b > a$, and it must be that $w(a) < n-a+1$. Thus, we have $B_{a, w(a)}\neq 0$ and $B_{b, w(b)}\neq 0$ with $a < b$ and $w(a) < w(b)$. By construction of $A'$, all entries $B_{i,j}$ with $a \leq i \leq b$ and $w(a) \leq j \leq w(b)$ are nonzero. Thus, the entries $B_{a, n-a+1}$ and $B_{b, w(a)}$ are both nonzero. The monomials $m_w$ and $m_{t_{a,b}w}$ only differ in these two factors, hence $t_{a,b}$ satisfies the conditions (1) and (2) described above. Now, $t_{a,b}w(i) = w_0(i)$ for all $1\leq i \leq a$, and we may proceed as before.
\end{proof}

\section{Proof of Conjecture for $q = 2$}\label{sec: d = 2 proof}

In this section, we will prove Conjecture \ref{conj: determinantal ideal is elimination ideal} for $q = 2$. Note that when $q = 0$, the statement is trivial, and when $q = 1$, the statement is proven in \cite[Thm 1.1]{transfer-paper}. Fixing $q = 2$, the polynomials $f_0,\ldots,f_{p-1}$ are given by
\begin{align*}
    f_0 &= t^2 + e_pt + e_{2p}\\
    f_1 &= e_1t + e_{p+1} \\
    f_2 &= e_2t + e_{p+2} \\
    &\vdots \\
    f_{p-1} &= e_{p-1}t + e_{p + (p-1)}
\end{align*}
and the matrix $A$ is given by
$$
A = \begin{pmatrix}
       1 & e_1 & 0 & e_2 & 0 & \cdots & e_{p-1} & 0 \\
       e_p & e_{p+1} & e_1 & e_{p+2} & e_2 & \cdots & e_{p + (p-1)} & e_{p-1} \\
       e_{2p} & 0 & e_{p+1} & 0 & e_{p+2} & \cdots & 0 & e_{p+ (p-1)}
   \end{pmatrix}.
$$

\begin{notation}\label{notation for all ideals}
    Fix $p\geq 3$ and let $S = \kk[e_1,\ldots, e_{2p}]$. Let $f_0, f_1, \ldots, f_{p-1}$ and $A$ be as above. Consider the following ideals $I,J,L$ of $S$:
    \begin{align*}
        I &:= \langle f_0, f_1, \ldots, f_{p-1} \rangle \cap S\\
        J &:= \mathcal{I}_{3}(A) \quad (\text{the ideal generated by the maximal minors of }A) \\
        L &:= \langle e_1,\ldots, e_{p-1} \rangle \cap
        \langle e_2, \ldots, e_p \rangle \cap \cdots \cap
        \langle e_{p+1}, \ldots, e_{2p-1}\rangle\\
        &= \bigcap_{i=1}^{p+1} \langle e_i, e_{i+1}, \ldots, e_{i+p-2}\rangle.
    \end{align*}
\end{notation}

By Proposition \ref{prop: J contained in I}, we have that $J\subseteq I$, hence there is also a containment of initial ideals $\In(J)\subseteq \In(I)$ for any monomial order on $S$. By Proposition \ref{prop: lead term is antidiagonal}, $L\subseteq \In(J)$ under the grevlex monomial order. In this section, we will show that under a certain positive $\ZZ^p$-multigrading on $S$, there is an equality of multigraded Hibert series $\Hilb(L, \xx) = \Hilb(I, \xx)$. From this it would follow that $\In(I) = L$ and that $I = J$.  

To do so, consider the $\kk$-algebra homomorphism $\varphi: S \to \kk[u_1, \ldots, u_p,\alpha]$ defined by
$$
    \varphi(e_i) = 
    \begin{cases}
        u_i &\text{if }1 \leq i \leq p-1 \\
        u_p + \alpha & \text{if } i = p \\
        u_{i-p}\alpha & \text{if }p+1 \leq i \leq 2p
    \end{cases}
$$
\begin{prop}\label{prop: elim ideal as kernel}
    With $\varphi$ defined as above, $I = \ker(\varphi)$.
\end{prop}

\begin{proof}
    Recall that $I$ defines the set of coefficients for which $f_0,f_1,\ldots,f_{p-1}$ have a common root, which occurs if and only if there exist $\alpha,u_1,\ldots,u_p\in \kk$ such that the $f_i$ factor as
    \begin{align*}
        f_0 &= (t+\alpha)(t+u_p)\\
        f_1 &= (t+\alpha)u_1\\
        &\vdots\\
        f_{p-1} &= (t+\alpha)u_{p-1}.
    \end{align*}
    Hence, the $\{f_i\}$ having a common root is equivalent to their coefficients $(e_1,\ldots,e_{2p})$ being in the image of the map
    \begin{align*}
        \mathbb{A}^{p+1} &\to \mathbb{A}^{2p}\\
        (u_1,\ldots,u_p, \alpha) &\mapsto (u_1,\ldots,u_{p-1},u_p + \alpha, u_1\alpha,\ldots,u_p\alpha).
    \end{align*}
    Hence the defining ideal of the image must be $I= \ker\varphi$.
\end{proof}

Let $A = \mathrm{Im}(\varphi)$. We will show the equality of Hilbert series $\Hilb(S/L,\xx) = \Hilb(A,\xx)$ for some $\ZZ^p$-multigrading, which we describe now. Let $\epsilon_i$ denote the standard basis vector of $\ZZ^p$ which has 1 in position $i$ and 0 everywhere else. On $S$, we make a $\ZZ^p$-multigrading by setting
$$
\deg(e_i) = \begin{cases}
    \epsilon_i &\text{if }1 \leq i \leq p\\
    \epsilon_{i-p} + \epsilon_p &\text{if } p+1 \leq i \leq 2p.
\end{cases}
$$
On $\kk[u_1,\ldots,u_p,\alpha]$ we set $\deg(u_i) = \epsilon_i$ and $\deg(\alpha) = \epsilon_p$. Under this multigrading, the map $\varphi$ is homogeneous. Now, fix any monomial ordering $\prec$ on $A$ in which $\alpha \prec u_p$.

\begin{lemma}\label{lemma: elements in initial algebra}
    For each $i = 1, \ldots, p-1$, the monomials $$u_i\alpha^j, \quad j \geq 1$$ appear in $A$ (and hence in the initial algebra $\In(A)$).
\end{lemma}

\begin{proof}
    Fix $1 \leq i \leq p-1$. Then 
    $$u_i\alpha^2 = (u_p + \alpha)\cdot u_i\alpha - u_i\cdot u_p\alpha\quad  \in A,$$
    and by induction one can show that $u_i\alpha^j \in A$ for all $j\geq 1$.
\end{proof}

\begin{cor}\label{cor: dim of graded pieces}
    Let $\dd = (d_1,\ldots,d_p)\in\ZZ_{\geq0}^p$. Then
    $$
    \dim_\kk(\In(A))_{\dd} =
    \begin{cases}
        d_p + 1 & \text{if }d_i > 0\text{ for some }1\leq i \leq p-1\\
        1 + \left\lfloor\dfrac{d_p}{2}\right\rfloor & \text{if }d_i=0\text{ for all }1\leq i \leq p-1
    \end{cases}
    $$
\end{cor}

\begin{proof}
    First, note that the given multigrading on $\kk[u_1,\ldots,u_p,\alpha]$ restricts to the fine grading on $u_1,\ldots,u_{p-1}$, in the sense that for each $(d_1,\ldots,d_{p-1})\in\ZZ_{\geq0}^{p-1}$, there is exactly one monomial in $u_1,\ldots,u_{p-1}$ of this multidegree.

    Now, fix a multidegree $\dd = (d_1,\ldots,d_p)\in\ZZ^p$ and suppose that at least one of $d_1,\ldots,d_{p-1}$ is nonzero. Without loss of generality, suppose $d_1 > 0$. Then we claim that all monomials of degree $\dd$ from $\kk[u_1,\ldots,u_p,\alpha]$ are present in $\In(A)$. Any such monomial is of the form $$u_1^{d_1}u_2^{d_2}\cdot \cdots \cdot u_{p-1}^{d_{p-1}} \cdot \alpha^e u_p^{d_p - e}.$$ Since $d_1\geq1$, by Lemma \ref{lemma: elements in initial algebra}, the given monomial may be factored in $\In(A)$ as
    $$(u_1\alpha^e)\cdot u_1^{d_1-1}\cdot u_2^{d_2}\cdots u_{p-1}^{d_{p-1}}\cdot u_p^{d_p-e}.$$
    The number of such monomials is equal to the number of (standard graded) monomials in $\{\alpha, u_p\}$ of degree $d_p$, which is $d_p+1$. This proves the first case.

    Now suppose that $d_i=0$ for all $i\neq p$. Since $\alpha^k\notin \In(A)$ for any $k\geq1$, any monomial of degree $(0,\ldots,0,d_p)$ lying in $\In(A)$ must be a monomial in $\{u_p, u_p\alpha\}$. Specifically, this set of monomials is
    $$\left\{ u_p^{d_p}, u_p^{d_p-2}(u_p\alpha), u_p^{d_p-4}(u_p\alpha)^2,\ldots,u_p^{d_p - 2\lfloor\frac{d_p}{2}\rfloor}(u_p\alpha)^{\lfloor\frac{d_p}{2}\rfloor}\right\},$$
    the cardinality of which is $1 + \lfloor\frac{d_p}{2}\rfloor$.
\end{proof}

Recall that for any homogeneous ideal $\mathfrak{a}$ of a positively graded polynomial ring $T$ over a field $\kk$, the quotient $T/\mathfrak{a}$ has a homogeneous $\kk$-basis of \textit{standard monomials} $\{m : m\notin \In(\mathfrak{a})\}$, where $\In(\mathfrak{a})$ may be with respect to any fixed monomial order on $T$. Since $L\subseteq \In(I)$, given any multidegree $\dd\in\ZZ^p$, we have
$$\dim_\kk(S/L)_{\dd} \geq \dim_\kk(S/\In(I)) =  \dim_\kk(S/I)_{\dd} = \dim_\kk(\In(A))_{\dd}.$$
Thus, if we can exhibit surjective linear maps 
$$\In(A)_{\dd}\to (S/L)_{\dd}$$
for all multidegrees $\dd$, it would follow that $\In(A)$ and $S/L$ have the same Hilbert series.

Now that we have characterized the monomials $m$ in $\In(A)_{\dd}$, we will create a map of vector spaces $\psi:\In(A)_{\dd}\to (S/L)_{\dd}$ by first constructing a factorization $x(m)$ of the form $m = x_1\cdot x_2 \cdots x_k$, where each $x_i\in \{u_1,\ldots, u_p, u_1\alpha,\ldots, u_p\alpha, \alpha\}$. The map $\psi$ will then be defined on the factors and extended multiplicatively to monomials.

\begin{construction}\label{const: factorization}
Fix a multidegree $\dd = (d_1,\ldots,d_p)\in\mathbb{N}^p$ and let $m$ be a monomial in $\In(A)_{\dd}$. We construct a factorization $x(m) = x_1\cdots x_k$ as follows.
\begin{enumerate}
    \item Set $i = p$, and let $\xx= ()$ be the empty sequence.
    \item While $u_i\alpha$ divides $m$, append $u_i\alpha$ to $\xx$ and replace $m$ by $\frac{m}{u_i\alpha}$.
    \item If $m = 1$, the algorithm finishes. Otherwise, replace $i$ by $i-1$. If $i = 0$, proceed to (4), otherwise return to (2).
    \item If $i = 0$, then $m = \alpha^\ell$ for some $\ell$ or $m = u_1^{a_1}\cdots u_p^{a_p}$ for some $a_1,\ldots,a_p$. Append $(\alpha,\ldots,\alpha)$ (with $\alpha$ repeated $\ell$ times) or $(u_1,\ldots,u_1,u_2,\ldots,u_p)$ (with $u_i$ repeated $a_i$ times) to $\xx$, respectively, and exit.
\end{enumerate}

Let $X= \{u_1,\ldots,u_p,u_1\alpha,\ldots,u_p\alpha,\alpha\}$ denote the set of possible factors $x_i$ in a factorization. Define a map of sets 
\begin{align*}
    \psi': X &\to S\\
    u_i &\mapsto e_i \\
    u_i\alpha &\mapsto e_{p+i} \\
    \alpha &\mapsto e_p
\end{align*}
We then define the map of vector spaces $\psi: \In(A)\to S$ by $\psi(m) = \prod_{i-1}^{k}\psi'(x_i)$ on monomials $m$ and extended $\kk$-linearly. Here, $m$ is factored as $m = x_1\cdots x_k$ according to Construction \ref{const: factorization}. Note that $\psi$ is a degree-preserving map of sets, so descends to maps $\psi:\In(A)_{\dd}\to S_{\dd}$ for all $\dd$.
\end{construction}

\begin{example}
    Let $p = 3$ and let $\dd = (1,1,3)$. We list all monomials in $\In(A)$ of degree $\dd$, followed by their factorizations and images under $\psi$.
    \begin{align*}
        u_1u_2u_3^3 &= u_3\cdot u_3\cdot u_3\cdot u_2\cdot u_1 \mapsto e_1e_2e_3^3\\
        u_1u_2u_3^2\alpha &= (u_3\alpha) \cdot u_3\cdot u_2\cdot u_1 \mapsto e_1e_2e_3e_6 \\
        u_1u_2u_3\alpha^2 &= (u_3\alpha)\cdot (u_2\alpha) \cdot u_1 \mapsto e_1e_5e_6 \\
        u_1u_2\alpha^3 &= (u_2\alpha)\cdot(u_1\alpha)\cdot \alpha \mapsto e_3e_4e_5
    \end{align*}
\end{example}

\begin{prop}
    The image of $\psi$, as defined in Construction \ref{const: factorization}, lies in the $\kk$-span of monomials \textit{not} in $L$, hence gives a map $\psi:\In(A)_{\dd} \to (S/L)_{\dd}$ for all multidegrees $\dd$.
\end{prop}

\begin{proof}
    Recall that a monomial $m$ of $S$ lies in $L$ if and only if it has no large gap (see Definition \ref{def: large gap}). We will show that any monomial in the image of $\psi$ has a large gap. Given a monomial $n\in \In(A)$, let $i$ be the smallest index in which $u_i\alpha$ appears in the factorization $x(n)$. Then one has the following possibilities, based on Construction \ref{const: factorization}:
    \begin{enumerate}
        \item[(i)] Suppose that one exited the algorithm in Construction \ref{const: factorization} at step (3). Then all factors in $x(n)$ are of the form $u_k\alpha$ for some $k$'s, and $\psi(n)\notin \langle e_1,\ldots,e_{p-1}\rangle \supset L$.
        \item[(ii)] If one exits the algorithm at step (4) with $n = \alpha^\ell$, then all factors in $x(n)$ are either some $u_k\alpha$ or $\alpha$. Then $\psi(n)$ also does not lie in $\langle e_1,\ldots,e_{p-1}\rangle\supset L$.
        \item[(iii)] If one exits the algorithm at step (4) with $n$ being a monomial in the $u_k$, then the largest $k$ for which the exponent on $u_k$ is nonzero has $k \leq i$. Thus, $\psi(n)$ has a gap between $e_{p+i}$ and $e_k$, with $k\leq i$. Consequently $\psi(n)$ does not lie in $\langle e_{i+1},\ldots, e_{p+i-1}\rangle \supset L$.
    \end{enumerate}
\end{proof}

We now define a right inverse for $\psi$. For an element $f\in A$ (with monomial order satisfying $u_p\succ \alpha$), define $\In(f)$ to be its initial monomial, and define $\In^\vee(f)$ to be the smallest monomial of $f$.

Let $\phi: (S/L)_{\dd} \to \In(A)_{\dd}$ be defined by
$$\phi(m) =
\begin{cases}
    \In(\varphi(m)) & \text{if }m\in \langle e_1,\ldots,e_{p-1} \rangle \\
    \In^\vee(\varphi(m))& \text{otherwise}
\end{cases}.
$$
Note that the two cases agree if $e_p$ does not divide $m$, as $\varphi(m)$ would be a monomial.

\begin{prop}\label{prop: right inverse}
    The map $\phi$ is a right inverse for $\psi$, i.e. $\psi\circ \phi = \mathrm{Id}_{(S/L)_{\dd}}$.
\end{prop}

\begin{proof}
 We compute the composition $\psi\circ\phi$ directly. Given $m\in (S/L)_{\dd}$. If $m \in \langle e_1,\ldots,e_{p-1}\rangle$, then $m = e_1^{a_1}\cdots e_{2p}^{a_{2p}}$ for some exponents $a_i$, with $a_i>0$ for some $i = 1, \ldots, p-1$. Applying $\phi$, we obtain 
 $$\phi(m) = u_1^{a_1} \cdots u_p^{a_p}\cdot(u_1\alpha)^{a_{p+1}}\cdots (u_{2p}\alpha)^{a_{2p}}.$$ The grouping written above is exactly the factorization $x(\phi(m))$, hence applying the map $\psi$ will recover $m$.

 Now suppose that $m\notin \langle e_1,\ldots, e_{p-1}\rangle$. Then $m = e_p^{a_p}\cdots e_{2p}^{a_{2p}}$ for some exponents $a_i$. Applying $\phi,$ we obtain
 $$\phi(m) = \alpha^{a_p}\cdot(u_1\alpha)^{a_{p+1}}\cdots (u_{2p}\alpha)^{a_{2p}},$$
 and again the above grouping is the factorization $x(\phi(m))$. Applying $\psi$ recovers $m$.
\end{proof}

\begin{cor}\label{cor: I is ideal of minors}
   When $q = 2$, Conjecture \ref{conj: determinantal ideal is elimination ideal} holds. That is, the elimination ideal $I = \langle f_0, \ldots, f_{p-1}\rangle \cap \kk[e_1,\ldots,e_{2p}]$ is equal to the determinantal ideal 
   $$J = \mathcal{I}_3
   \begin{pmatrix}
       1 & e_1 & 0 & e_2 & 0 & \cdots & e_{p-1} & 0 \\
       e_p & e_{p+1} & e_1 & e_{p+2} & e_2 & \cdots & e_{p + (p-1)} & e_{p-1} \\
       e_{2p} & 0 & e_{p+1} & 0 & e_{p+2} & \cdots & 0 & e_{p+ (p-1)}
   \end{pmatrix},
   $$
   where $\mathcal{I}_3(-)$ denotes taking the ideal generated by the $3\times3$ minors of the specified matrix.
\end{cor}

\begin{proof}
    By Proposition \ref{prop: J contained in I}, $J\subseteq I$, hence we also have $\In(J)\subseteq\In(I)$. By Proposition \ref{prop: lead term is antidiagonal}, the ideal $L$ of Notation \ref{notation for all ideals} is contained in $\In(J)$. Proposition \ref{prop: right inverse} shows the equality of multigraded Hilbert series
    $$\Hilb(S/I,\mathbf{x}) = \Hilb(S/L,\mathbf{x}),$$
    meaning that $L = \In(I) = \In(J)$. This means that $I$ and $J$ also have the same multigraded Hilbert series, hence they are equal. 
\end{proof}

\section{Syzygies in the quadratic and linear case}\label{sec: d = 2 resolution}

In this section, we exhibit an explicit minimal free resolution of the associated graded ideal of $I$ when $q = 2$. This minimal free resolution is linear with respect to the standard grading on $S$, from which we can conclude that the total Betti numbers of $I$ and its associated graded ideal agree.

We now fix $q = 2$, $p\geq 3$, and $f_0,f_1,\ldots,f_{p-1}$ as in (\ref{eq: f_1..f_k}). By Corollary \ref{cor: I is ideal of minors}, the elimination ideal $I = \langle f_0, f_1,\ldots, f_{p-1}\rangle \cap \kk[e_1,\ldots,e_{2p}]$ is generated by the $3\times3$ minors of the matrix

\begin{equation}\label{eq:matrix A}
 A=
\begin{pmatrix}
1  & e_1 & 0 & e_2 & 0 & \cdots & e_{p-1} & 0\\
e_p & e_{p+1} & e_1 & e_{p+2} & e_2 & \cdots & e_{p + (p-1)} & e_{p-1} \\
e_{2p} & 0 & e_{p+1} & 0 & e_{p+2} & \cdots & 0 & e_{p + (p-1)}
\end{pmatrix}.
\end{equation}

\subsection{Establishing the associated graded ideal}

Let us now consider $S = \kk[e_1,\ldots,e_{2p}]$ as a standard graded $\kk$-algebra with $\deg(e_i) = 1$. Let $\mm=(e_1,\ldots,e_{2p})$ be the homogeneous maximal ideal, and let $I$ be the inhomogeneous ideal $\mathcal{I}_3(A)$ generated by the $3\times 3$ minors of $A$. Following \cite[15.10.3]{EisenbudCA}, the \textit{associated graded ring} of $S/I$ with respect to the $\mm$-adic filtration takes the form $S/I'$, where $$I' = \langle \In(f):f\in I\rangle.$$ is the ideal generated by initial forms of elements of $I$. Here, the initial form $\In(f)$ of $f\in S$ is computed by taking the image of $f$ in $\mm^i/\mm^{i+1}$, where $i$ is the largest index for which $f\in \mm^i\setminus \mm^{i+1}$. In other words, $\In(f)$ can be represented by taking the sum of lowest (standard) degree terms of $f$. 

By \cite[Prop 15.28]{EisenbudCA}, one can compute a generating set for $I'$ by first constructing the ideal $\tilde{I}\subset S[t]$, which is generated by homogenizations of generators of $I$, where the variable $t$ of degree 1 is used to homogenize. Then, one computes a Gr\"obner basis for $\tilde{I}$ with respect to any monomial order on $S[t]$ which refines the partial order by $t$-degree. Finally, one dehomogenizes the Gr\"obner basis and takes initial forms.

\begin{prop}\label{prop: 3x3 minors are gb}
    Let $\cG$ denote the set of maximal minors of $A$ of the following two types:
    \begin{enumerate}
    \item[(i)] $
    \det
    \begin{pmatrix} 1 & 0 & 0 \\
    e_p & e_{i} & e_j \\
    e_{2p} & e_{p+i} & e_{p+j} \end{pmatrix} = 
    -e_{p+i}e_j + e_{p+j}e_i
    $,\,\,\, where $1 \leq i < j \leq p-1$, and \\[5pt]
    \item[(ii)] $\det
    \begin{pmatrix}
        1 & e_i & 0 \\
        e_p & e_{p+i} & e_j \\
        e_{2p} & 0 & e_{p+j}
    \end{pmatrix} = e_{p+i}e_{p+j}- e_{i}e_{p+j}e_p + e_ie_je_{2p}
    $ where $1\leq i\leq j \leq p-1$.
\end{enumerate}
Consider the monomial order on $S[t]$ which first compares degree on $t$, then breaks ties using the grevlex order on $S$. Under this monomial order, the homogenizations of elements of $\mathcal{G}$ form a Gr\"obner basis for $\tilde{I}$.
\end{prop}

\begin{proof}
    We first note that $\cG$ generates $I$. If a chosen maximal submatrix of $A$ has two zeroes in the same row, then its determinant is a multiple of a minor of type (i). Otherwise, it is a determinant of type (ii) but with $i > j$ and can be rewritten using the relation
    $$
    e_{p+i}e_{p+j}- e_{j}e_{p+i}e_p + e_ie_je_{2p} = (e_{p+i}e_{p+j}- e_{i}e_{p+j}e_p + e_ie_je_{2p}) - e_p(e_{p+i}e_j - e_{p+j}e_i).
    $$
    Let $\tilde{\cG}$ denote the set of polynomials in $S[t]$ which are the homogenizations of elements of $\cG$. By definition, $\tilde{\cG}$ generates $\tilde{I}$.
    By Buchberger's criterion, $\tilde{\cG}$ is a Gr\"obner basis if and only if all $S$-pairs among elements of $\tilde\cG$ reduce to 0 mod $\tilde\cG$. More specifically, one needs to check all $S$-pairs $S(f_{n, i, j}, f_{m, k, \ell})$, where $f_{n,i,j}$ denotes the homogenization of the minor of type $n$ with parameters $i$ and $j$. Here, $n$ is (i) or (ii), and $1\leq i,j \leq p-1$. The $S$-pair will depend on the relative orders of $i,j,k,\ell$. All possible relative orders are present when $p = 5$, hence $\tilde\cG$ is a Gr\"obner basis for $p = 5$ if and only if it is a Gr\"obner basis for all $p$. One can verify by computer, e.g. via \texttt{Macaulay2}, that $\tilde\cG$ does form a Gr\"obner basis for $p = 5$.
\end{proof}

\begin{cor}\label{cor: weight vector}
    The associated graded ideal $I'$ is equal to the sum of ideals 
    $$I' = \mathcal{I}_2
    \begin{pmatrix}
        e_1 & e_2 & \cdots & e_{p-1}\\
        e_{p+1} & e_{p+2} & \cdots & e_{p + (p-1)}
    \end{pmatrix} + \langle e_{p+1}, e_{p+2}, \ldots e_{p + (p-1)}\rangle^2.
    $$
\end{cor}

\begin{proof}
    By Proposition \ref{prop: 3x3 minors are gb}, generators for $I'$ can be computed by taking initial forms of the minors of types (i) and (ii). These initial forms are
     \begin{align*}
         \In(e_{p+i}e_j - e_{p+j}e_i) &= e_{p+i}e_j - e_{p+j}e_i \text{ for all }i< j, \text{ and }\\
         \In(e_{p+i}e_{p+j} - e_{i}e_{p+j}e_p + e_ie_je_{2p})&= e_{p+i}e_{p+j} \text{ for all }i,j.
     \end{align*}
     Hence we obtain the desired initial form ideal.
\end{proof}

\begin{remark}
    We note that when $p = 3$, the ideal $I'$ appears in work of Mantero--Mastroeni \cite[Thm A]{mantero-mastroeni}, where the authors classify Koszul algebras generated by 4 quadrics. The resolution described there also agrees with the equivariant resolution constructed in Section \ref{sec: d = 2 resolution}.
\end{remark}

\subsection{Equivariant resolution of the initial ideal}
Consider the short exact sequence of $S$-modules
\begin{equation}\label{eq: ses}
    0 \to K \to S/\nn^2 \to S/I' \to 0,
\end{equation}
where $\nn = \langle e_{p+1}, e_{p+2}, \ldots, e_{p+(p-1)}\rangle$ and $K$ denotes the kernel of the surjection $S/\nn^2 \to S/I'$. We will use a mapping cone to construct a minimal resolution of $S/I'$ from resolutions of $K$ and $S/\nn^2$. 

\begin{notation}\label{notation: V W}
    Given a finite dimensional $\kk$-vector space $U\cong \kk^m$, let $\GL(U)\cong \GL_{m}$ denote the general linear group of invertible $\kk$-linear maps $U\to U$. If $\lambda = (\lambda_1,\lambda_2,\ldots,\lambda_k)$ is a partition with $\lambda_1 \geq \lambda_2 \geq \cdots \geq \lambda_k> 0$, we denote by $\SS^\lambda U$ the Schur functor applied to $U$. We use the characteristic free definition of Schur functors found in \cite{ABW}. The vector space $\SS^\lambda(U)$ has a $\kk$-basis of semistandard Young tableaux of shape $\lambda$ and entries $\leq m$. For example, if $m = 3$ and $\lambda = (2,1)$, a basis for $\SS^{\lambda}(U)$ is
    $$
    \left\{
    \begin{ytableau}
        1 & 1 \\2
    \end{ytableau},\,\,
    \begin{ytableau}
        1 & 1 \\3
    \end{ytableau},\,\,
    \begin{ytableau}
        1 & 2 \\2
    \end{ytableau},\,\,
    \begin{ytableau}
        1 & 2 \\3
    \end{ytableau},\,\,
    \begin{ytableau}
        2 & 2 \\3
    \end{ytableau},\,\,
    \begin{ytableau}
        2 & 3 \\3
    \end{ytableau}
    \right\}
    $$
    The symmetric algebra on $U$ will be denoted by $S(U)$ and has graded decomposition $$S(U)= \bigoplus_{d\geq0}S_d(U),$$
    where $S_d(U)$ denotes the $d^{\text{th}}$ symmetric power of $U$. We next define some $\GL(U)$-equivariant maps. All tensor symbols $\otimes$ are $\otimes_\kk$, unless otherwise specified. For each $k\geq 2$, let  $$\Delta: \textstyle{\bigwedge^kU\to \bigwedge^{k-1}U\otimes U}$$ be the map $\Delta(u_{i_1}\wedge u_{i_2}\wedge\cdots\wedge u_{i_k}) = \sum_{j=1}^k (-1)^j u_{i_j}\otimes u_{i_1}\wedge\cdots\wedge\widehat{u_{i_j}}\wedge\cdots\wedge u_{i_k}$. Similarly, let $$m:S_{k-1}(U)\otimes U \to S_k(U)$$ denote multiplication in the symmetric algebra, where we identify $U$ with $S_1(U)$.

    Let $V$ and $W$ be the $\kk$-vector spaces with bases $\{e_1,e_2\ldots, e_{p-1}\}$ and $\{e_{p+1}, e_{p+2}, \ldots, e_{p+(p-1)}\}$, respectively. Let $S' = S(V)\otimes S(W)$. Any map of the form $\SS^\lambda V\to \SS^\lambda W$ or $S'\otimes \SS^\lambda V  \to S' \otimes \SS^\mu W$ is precomposed with the isomorphism $\SS^\lambda V \to \SS^\lambda W$ induced from the isomorphism $V\to W, e_i \mapsto e_{p+i}$.
\end{notation}

\subsubsection{Resolution of the kernel}
\label{subsubsec: res of K}First we discuss the minimal resolution of $K$, which arises from a double complex built from two Koszul complexes. We take a moment here to recall details on Koszul complexes and establish notation.

\begin{defn}\label{def: koszul complex}
    Let $U$ be a finite-dimensional $\kk$-vector space $U$ with basis $\{u_1,\ldots,u_r\}$. Define the \textbf{Koszul complex} $(\cK^U, \partial^U)$ to be the chain complex of $S(U)$-modules with terms $$\cK^U_i = S(U)\otimes\textstyle{\bigwedge^i(U)}$$ and differential
    \begin{align*}
        \partial^U_i: S(U)\otimes\textstyle{\bigwedge^iU}&\longrightarrow S(U)\otimes\textstyle{\bigwedge^{i-1}U}\\
        s\otimes u_{j_1}\wedge \cdots \wedge u_{j_i} &\mapsto \sum_{k = 1}^i (-1)^k su_{j_k} \otimes u_{j_1} \wedge \cdots \wedge \widehat{u_{j_k}} \wedge\cdots \wedge u_{j_i}.
    \end{align*}
The differential can be described in a coordinate-free way as the composition
\begin{center}
\begin{tikzcd}
    S(U)\otimes \bigwedge^iU \arrow[r, "1\otimes\Delta"]&
    S(U)\otimes U \otimes \bigwedge^{i-1}U \arrow[r, "m\otimes 1"] & 
    S(U) \otimes \bigwedge^{i-1}U.
\end{tikzcd}
\end{center}
Since $(u_1,\ldots,u_r)$ forms a regular sequence in $S(U)$, the complex $\cK^U$ satisfies $H_i(\cK^U) = 0$ if $i>0$, and $H_0(\cK^U) = S(U)/\langle u_1,\ldots,u_r\rangle$. 
\end{defn}

We now construct our resolution $\cF$ of $K$. We work over $S'$ rather than $S$, noting that the extension $S'\to S$ is flat and will not change homology. Adopt Notation \ref{notation: V W} and define a double complex $\mathcal{C}$ of $S'$-modules by setting 
\begin{align*}
    C_{p,q} &= S' \otimes \textstyle{\bigwedge^{p+2}V}\otimes \textstyle{\bigwedge^qW} \\
    \partial^v &=  1\otimes (-1)^p\partial^W_q:C_{p,q}\to C_{p, q-1}\\
    \partial^h &= \partial^V_p\otimes 1:C_{p,q}\to C_{p-1,q}.
\end{align*}
This is indeed a double complex; one can verify directly that 
$0 = (\partial^v)^2 = (\partial^h)^2 = \partial^v\partial^h + \partial^h\partial^v$. Let $\cF$ denote the total complex $\Tot(\cC)$ with terms $$\cF_n = \bigoplus_{p+q = n}C_{p,q}$$ and differential $$\partial = \partial^v + \partial^h.$$ To any double complex, one can associate a column filtration on its total complex, and this gives rise to a spectral sequence $E^r_{p,q}$ (for further details, see \cite[Ch. 5]{Weibel}). Because $\cC$ is a double complex supported only in the first quadrant (that is, $C_{p,q}=0$ if $p<0$ or $q<0$), this spectral sequence converges to the homology of the total complex:
$$E^2_{p,q}\implies H_{p+q}(\cF).$$
We analyze this spectral sequence to compute the homology of $\cF$. 

\begin{prop}\label{prop: double complex acyclic}
    The total complex $\cF$ satisfies $H_i(\cF) = 0$ unless $i = 0$.
\end{prop}

\begin{proof}
    Consider the spectral sequence associated to the column filtration of $\cF = \Tot(\cC)$. The $E^0$ page of has terms $E^0_{p,q} = C_{p,q}$, and the differentials coincide with the vertical differentials $\partial^v$. The $p^{\text{th}}$ column of the $E^0$ page is thus the Koszul complex $\cK^W$ of $S(W)$-modules tensored with $S(V)\otimes \bigwedge^{p+2}V$. In other words, 
    \begin{align*}
        E^0_{p,*} &= \textstyle{\bigwedge^{p+2}V}\otimes  S(V) \otimes\cK^W \\
        &= \textstyle{\bigwedge^{p+2}V}\otimes S(V) \otimes S(W)\otimes_{S(W)} \cK^W\\
        &= \textstyle{\bigwedge^{p+2}V} \otimes S' \otimes_{S(W)} \cK^W.
    \end{align*}
   %$$E^0_{p,*} = \cK^W\otimes_{S(W)}S'\otimes_{S'}\bigwedge^{p+2}V = \cK^W \otimes_{S(W)}\bigwedge^{p+2}V.$$
    Recall that $\cK^W$ only has homology in degree 0. Since extending scalars to $S'$ and tensoring with a free module are both flat functors, we conclude that $H_q(E^0_{p,*}) = 0$ unless $q=0$. Hence, the $E^1$ page of the spectral sequence is concentrated in the single row $q = 0$, with differentials induced from $\partial^h$. 
    
    Let $\nn = (e_{p+1}, \ldots, e_{p + (p-1)})\subset S(W)$. Since $H_0(\cK^W)\cong S(W)/\nn$, the nonzero row $E^1_{p,0}$ has the form
    \begin{center}
        \begin{tikzcd}
        S(V)\otimes\bigwedge^2V\otimes S(W)/\nn &
        S(V)\otimes\bigwedge^3V\otimes S(W)/\nn \arrow[l, "\partial^V\otimes 1" above] &
        S(V)\otimes\bigwedge^4V\otimes S(W)/\nn \arrow[l, "\partial^V \otimes 1" above]&
        \cdots \arrow[l, "\partial^V\otimes 1" above]
        \end{tikzcd}
    \end{center}
    or equivalently
    \begin{center}
        \begin{tikzcd}
            S'/\nn S'\otimes \bigwedge^2V &
            S'/\nn S' \otimes \bigwedge^3V \arrow[l, "\partial^V" above] &
            S'/\nn S' \otimes \bigwedge^4V \arrow[l, "\partial^V" above] &
            \cdots \arrow[l, "\partial^V" above]
        \end{tikzcd}
    \end{center}
    Now, let $M$ be the $S'$-module $\operatorname{coker}(S'\otimes\bigwedge^3V\overset{\partial^V}{\longrightarrow} S'\otimes\bigwedge^2V)$. The truncated Koszul complex $\cL$ whose $i^{\text{th}}$ term is $S'\otimes \bigwedge^{i+2}V$ gives a free resolution of $M$ over $S'$. The sequence $(e_{p+1}, \ldots, e_{p+(p-1)})$ is regular on $M$ and on $S'$. Hence, $\cL\otimes S'/\nn$ is a free resolution of $M/\nn M$ over $S'/\nn$. In particular, $H_i(\cL\otimes S'/\nn)=0$ if $i>0$. But $\cL\otimes S'/\nn$ is exactly the complex appearing in the $0^{\text{th}}$ row of the $E^1$-page. Hence, the $E^2$-page is only nonzero if $(p,q) = (0,0)$. We therefore have $E^2_{p,q} = E^\infty_{p,q}$, and we conclude by convergence of the spectral sequence to $H_*(\cF)$ that $H_i(\cF) = 0$ if $i\neq 0$.
\end{proof}

\begin{prop}
    The total complex $\cF$ is a free resolution of $K$, that is, $H_0(\cF) \cong K$ as $S$-modules.
\end{prop}

\begin{proof}
    First, we show that $K\cong J/\nn J$ as $S$-modules, where $J$ is the ideal of $S$ generated by the $2\times 2$ minors of
    $$ B = 
    \begin{pmatrix}
        e_1 & e_2 & \cdots & e_{p-1} \\
        e_{p+1} & e_{p+2} & \cdots & e_{p + (p-1)}
    \end{pmatrix}.
    $$
    Since $K$ is the $S$-submodule of $S/\nn^2$ generated by the $2\times 2$ minors of $B$, there is a surjective map $J \to K$. Any $\nn$-multiple of a $2\times 2$ minor lies in $\nn^2$, hence is 0 in $K$. Thus, the map $J\to K$ factors through $J/\nn J$, giving a surjection $\pi:J/\nn J\to K$. We claim that $\pi$ is an isomorphism. Given 
    $$f = \sum_{1\leq i<j \leq p-1} s_{i,j} \det\begin{pmatrix} e_i & e_j \\ e_{p+i} & e_{p+j} \end{pmatrix},$$
    $\pi(f) = 0$ in $K$ if and only if $\pi(f)\in \nn^2$. With respect to the above expression, this means that either all $s_{i,j}$ lie in $\nn$ or that the entire expression is 0 in $J$. Hence, the map $\pi$ is injective out of $J/\nn J$.

    Now that we have $J/\nn J\cong K$, we may build a presentation for $J/\nn J$ that agrees with the presentation $\cF_1 \to \cF_0$. Since $J$ is generated as an $S$-module by the $2\times 2$ minors of $B$, there is a surjection
    \begin{align*}
        S'\otimes_\kk \textstyle{\bigwedge^2V} &\overset{\varphi_0}{\longrightarrow} J/\nn J \\
        e_i\wedge e_j &\mapsto \det\begin{pmatrix} e_i & e_j \\ e_{p+i} & e_{p+j} \end{pmatrix}.
    \end{align*}
    Moreover, $J$ is minimally resolved over $S$ by the Eagon--Northcott complex (see \cite[A2.6]{EisenbudCA}), giving a map 
    \begin{center}
        \begin{tikzcd}
            S' \otimes \left(\bigwedge^3V\oplus\bigwedge^3V\right) \arrow[r, "\varphi_{1,1}"] &
            S'\otimes\bigwedge^2V
        \end{tikzcd}
    \end{center} 
    whose image is in the kernel of $\varphi_0$. By construction of the Eagon--Northcott complex, this map can be represented by a block matrix $[X \,\,\, Y]$, where $X$ agrees with the Koszul differential $S'\otimes \bigwedge^3 V\to S'\otimes \bigwedge^2 V$ and $Y$ has all entries lying in $\nn$. More specifically, the columns of $X$ come from expanding determinants of the form
    $$\det\begin{pmatrix}
        e_i & e_j & e_k \\ e_i & e_j & e_k \\
        e_{p+i} & e_{p+j} & e_{p+k}
    \end{pmatrix}$$
    along the first row to give relations among $2\times 2$ minors. Similarly, the columns of $Y$ come from expanding determinants of the form 
    $$\det\begin{pmatrix}
        e_{p+i} & e_{p+j} & e_{p+k} \\ e_i & e_j & e_k \\
        e_{p+i} & e_{p+j} & e_{p+k}
    \end{pmatrix}$$
    along the first row.
    To capture the additional relations imposed by factoring out $\nn J$, we use the map
    \begin{center}
        \begin{tikzcd}
             S'\otimes \left(\bigwedge^2 V\otimes W\right) \arrow[r, "\varphi_{1,2}"] & S'\otimes\bigwedge^2V
        \end{tikzcd}
    \end{center}
    which multiplies $W$ into $S' \supseteq S(W)$. Then, the map
    \begin{center}
        \begin{tikzcd}
            \begin{gathered}
        S'\otimes\left( \textstyle{\bigwedge^3V} \oplus \textstyle{\bigwedge^3V} \right) \\
        \oplus \\
        S' \otimes \left(\textstyle{\bigwedge^2V} \otimes W \right)
    \end{gathered} \arrow[r, "\varphi_{1,1} | \varphi_{1,2}"] & S\otimes \textstyle{\bigwedge^2V}
        \end{tikzcd}
    \end{center}
    gives a presentation for $J/\nn J$. The matrix $Y$ in $\varphi_{1,1}$ is redundant, since any column of $Y$ is a linear combination of entries of $\varphi_{1,2}$. Hence we may replace $\varphi_{1,1}$ by the map 
    $$S\otimes \textstyle{\bigwedge^3V}\overset{X}{\longrightarrow} S\otimes \textstyle{\bigwedge^2V}.$$
    Upon making this replacement, the map $\varphi_1 = [X\,\,\, \varphi_{1,2}]$ is exactly the differential $\cF_1 \to \cF_0$. Hence, $H_0(\cF)\cong J/\nn J\cong K$.
\end{proof}

\subsubsection{Resolution of $S/\nn^2$}

The $S$-module $S/\nn^2$ has a well-known minimal free resolution $\cG$ with $\GL(W)$-equivariant structure; see \cite[Thm 3.1]{BuchsbaumEisenbud}. This resolution has the form
$$\cG: \qquad 0 \gets S(W) \gets S(W)\otimes \SS^{(2)}(W) \gets S(W)\otimes \SS^{(2,1)}(W) \gets S(W)\otimes \SS^{(2,1,1)}(W) \gets \cdots$$
with differential as follows. For $\lambda = (\lambda_1, \lambda_2, \ldots, \lambda_\ell)$, there is an inclusion of $\SS^\lambda(W)$ into the tensor product of symmetric powers of $W$ in the row sizes of $\lambda$: 
$$\SS^\lambda(W) \into S^{\lambda_\ell}(W)\otimes S^{\lambda_{\ell-1}}(W)\otimes \cdots \otimes S^{\lambda_1}(W).$$
Let $\partial:S(W)\otimes \SS^{(2, 1^{i+1})}(W) \to S(W)\otimes \SS^{(2, 1^i)}(W)$ be the composite map
$$
S(W) \otimes \SS^{(2, 1^{i+1})}(W) \into S(W)\otimes W \otimes W^{\otimes i} \otimes S^2(W) \to S(W) \otimes W^{\otimes i} \otimes S^2(W),$$
where the second map is multiplication of the first two tensor factors $S(W)\otimes W \to S(W)$. The image of this map is contained in $S(W)\otimes \SS^{(2, 1^i)}(W)$ and $\partial^2=0$; see, e.g., \cite[Sec 4.1]{Veronese} for an exposition and more general cases. In what follows, we use the complex of $S'$-modules $S'\otimes_{S(W)} \cG$.

\subsubsection{The comparison map}

We now give a chain map $\cF\to\cG$ which lifts the inclusion $K\to S/\nn^2$. Define $\varphi_0:\cF_0\to \cG_0$ to be the composite map
$$S'\otimes\textstyle{\bigwedge^2V} \cong S(V)\otimes\textstyle{\bigwedge^2V} \otimes S(W)\to S(V)\otimes V\otimes V\otimes S(W) \to S(V)\otimes S(W) = S'. $$

For each $i = 1, \ldots, p-1$, let $\varphi_i$ be the composite map
\begin{center}
    \begin{tikzcd}
        S \otimes \bigwedge^2V \otimes \bigwedge^iW \arrow[r, "1\otimes \Delta\otimes1"] &
        S\otimes V \otimes V \otimes \bigwedge^iW \arrow[r, "m\otimes\psi_i"]&
        S\otimes \SS^{(2, 1^{i-1})}W
    \end{tikzcd}
\end{center}
where $\Delta$ and $m$ are as in Notation \ref{notation: V W}, and $\psi_i:V\otimes\bigwedge^iW\to \SS^{(2, 1^{i-1})}W$ is an instance of the natural map 
$$\bigotimes_j \textstyle{\bigwedge^{\lambda'_j}U}\to \SS^\lambda U$$ for any vector space $U$ and partition $\lambda$ (where $\lambda'$ denotes the transpose partition). We let $\varphi:\cF\to \cG$ be the map which has components $\varphi_i:\cF_i\to\cG_i$ and is 0 on the summands in $\cF_i$ of the form $S\otimes \bigwedge^jV\otimes \bigwedge^{i+2-j}W$ for $j\neq2$.

\begin{prop}\label{prop: comparison map}
    The map $\varphi$ is a map of chain complexes.
\end{prop}

\begin{proof}
    We verify on the $S$-basis elements of $\cF_{i+1}$ that $\varphi$ commutes with the differentials of $\cF$ and $\cG$. Because the map $\varphi$ is 0 on the summands $S\otimes \bigwedge^kV\otimes \bigwedge^\ell W$ for which $k>2$, we only consider the basis elements in $S\otimes \bigwedge^2V\otimes\bigwedge^{i+1}W$. We use Young tableaux to represent elements of $\cF_k$ and $\cG_k$. Let $x = 1\otimes (e_a\wedge e_b) \otimes (e_{p+c_1}\wedge \cdots \wedge e_{p+c_{i+1}})$ be an $S$-basis element of $S\otimes \bigwedge^2V\otimes \bigwedge^{i+1}W\subset \cF_{i+1}$. We compute its image under $\varphi_i\circ\partial^\cF_{i+1} = \psi\circ\Delta\circ \partial^\cF_{i+1}$:
    \begin{align*}
        x = 1 \otimes \begin{ytableau}
            a \\ b
        \end{ytableau} \otimes
        \begin{ytableau}
            c_1 \\ \vdots \\ c_{i+1}
        \end{ytableau} &\overset{\partial^\cF_{i+1}}{\longmapsto} \sum_{j=1}^{i+1}(-1)^j e_{p+c_j}\otimes
        \begin{ytableau}
            a\\b
        \end{ytableau} \otimes
        \begin{ytableau}
            c_1 \\ \vdots \\ \hat{c_j} \\ \vdots \\ c_{i+1}
        \end{ytableau}\\
        &\overset{\Delta}{\longmapsto}
        \sum_{j=1}^{i+1}(-1)^j e_{p+c_j}\otimes(a\otimes b - b\otimes a) \otimes
        \begin{ytableau}
            c_1 \\ \vdots \\ \hat{c_j} \\ \vdots \\ c_{i+1}
        \end{ytableau}\\
        &\overset{\psi}{\longmapsto}
        \sum_{j=1}^{i+1}(-1)^j e_{p+c_j} \bigg(e_a\otimes
        \begin{ytableau}
            c_1 & b \\ \vdots \\ \hat{c_j} \\ \vdots \\ c_{i+1}
        \end{ytableau} - e_b\otimes\begin{ytableau}
            c_1 & a \\ \vdots \\ \hat{c_j} \\ \vdots \\ c_{i+1}
        \end{ytableau}\bigg).
    \end{align*}
    Next we give the image under $\partial^G_i\circ\varphi_{i+1} = \partial^G_i\circ\psi\circ\Delta:$
    \begin{align*}
        1 \otimes \begin{ytableau}
            a \\ b
        \end{ytableau} \otimes
        \begin{ytableau}
            c_1 \\ \vdots \\ c_{i+1}
        \end{ytableau} &\overset{\Delta}{\longmapsto}
        1 \otimes(a\otimes b - b\otimes a) \otimes 
        \begin{ytableau}
            c_1 \\ \vdots \\ c_{i+1}
        \end{ytableau} \\
        &\overset{\psi}{\longmapsto}
        e_a\otimes
        \begin{ytableau}
            c_1 & b \\ \vdots \\ c_{i+1}
        \end{ytableau} - e_b \otimes
        \begin{ytableau}
            c_1&a \\ \vdots \\ c_{i+1}
        \end{ytableau} \\
        &\overset{\partial^\cG_{i+1}}{\longmapsto}
        \sum_{j=1}^{i+1}(-1)^j e_ae_{p+c_j}\otimes 
        \begin{ytableau}
            c_1 & b \\
            \vdots \\ \hat{c_j} \\ \vdots \\ 
            c_{i+1}
        \end{ytableau} - 
        e_be_{p+c_j}\otimes
        \begin{ytableau}
            c_1 & a\\
            \vdots \\ \hat{c_j} \\ \vdots \\ c_{i+1}
        \end{ytableau}\\
        &= \varphi_i\circ \partial_{i+1}^{\cF}(x)
    \end{align*}
\end{proof}

\begin{thm}
    The mapping cone of $\varphi$, as defined just before Proposition \ref{prop: comparison map}, gives a minimal free resolution of $S/I'$. Moreover, there is an equality of projective dimensions $\mathrm{pdim}(S/I) = \mathrm{pdim}(S/I') = 2p-3$, and the total Betti numbers for $S/I$ and $S/I'$ are equal.
\end{thm}

\begin{proof}
    Since $0 \to K \to S/\nn^2 \to S/I'\to 0$ is short exact, and $\cF, \cG$ resolve $K$ and $S/\nn^2$, respectively, the mapping cone $\mathrm{Cone}(\varphi)$ will be a free resolution of $S/I'$. Let $S_+$ denote the ideal of all positive degree elements of $S$. Because $\varphi(\cF_i)\subset S_+ \cG_i$ for all $i$, the mapping cone gives a minimal resolution of $S/I'$. It was shown originally in \cite{Robbiano} and also proven in \cite[Thm 3.1/Cor 3.2]{HerzogRossiValla} that the Betti numbers for an associated graded ideal give an upper bound for the Betti numbers of the original ideal, i.e. $\beta^S_i(S/I)\leq \beta^S_i(S/I')$ for all $i$.
    
    From the construction of $\cF, \cG$, and the comparison maps, one can see that $\mathrm{Cone}(\varphi)$ gives a \textit{linear} free resolution of $S/I'$, meaning that all entries in the differential matrices are linear forms. By \cite[Theorem 3.1]{RossiSharifan} (see also \cite{Peeva-ConsecutiveCancel}), the total Betti numbers of $S/I$ may be obtained from the graded Betti numbers of $S/I'$ by a series of \textit{negative consecutive cancellations}. This corresponds to the presence of maps of the form $S(-j)\to S(-j')$, with $j'>j$, in the minimal free resolution of $S/I'$. However, since the minimal free resolution of $S/I'$ is linear, a negative consecutive cancellation is not possible. Hence, the total Betti numbers of $S/I$ and $S/I'$ must agree. The projective dimension of $S/I'$ can be computed directly from its resolution as the last nonzero degree in the resolution of $K$.
\end{proof}

\section*{Acknowledgments}
The first author was partially supported by NSF grant  DMS 2147769 and a Simons Fellowship. The second author is grateful to Paulo Mantero for helpful conversation regarding linear resolutions.

\bibliographystyle{plain}
\bibliography{biblio}

\end{document}